\documentclass[12pt]{amsart}

\usepackage{amsfonts}
\usepackage{amsmath}
\usepackage{amssymb}
\usepackage{amsthm}
\usepackage{enumerate} 
\usepackage{hyperref}
\hypersetup{
	colorlinks=true,
	citecolor=blue
}
\usepackage[margin=1.2in]{geometry}
\usepackage{mathrsfs} 
\usepackage{graphicx} 
\usepackage{tikz-cd}

\usepackage{xcolor}

\theoremstyle{plain}
\newtheorem{theorem}{Theorem}[section]
\newtheorem{lemma}[theorem]{Lemma}
\newtheorem{proposition}[theorem]{Proposition}

\theoremstyle{definition}

\newtheorem{remark}[theorem]{Remark}

\newtheorem{example}[theorem]{Example}
\newtheorem{examples}[theorem]{Examples}

\numberwithin{equation}{section}

\newcommand\N{\mathbb{N}}
\newcommand\Z{\mathbb{Z}}
\newcommand\R{\mathbb{R}}

\newcommand\C{\mathbb{C}}

\makeatletter
\@namedef{subjclassname@2020}{\textup{2020} Mathematics Subject Classification}
\makeatother

\title{Moment problems in the Schwartz and Gelfand-Shilov spaces}

\author[A. Debrouwere]{Andreas Debrouwere}
\address{Department of Mathematics and Data Science \\ Vrije Universiteit Brussel, Belgium\\ Pleinlaan 2 \\ 1050 Brussels \\ Belgium}
\email{andreas.debrouwere@vub.be}

\begin{document}

\begin{abstract}
We provide a geometric characterization of the closed sets $K \subseteq \mathbb{R}^d$ such that  \emph{every} real $d$-sequence is the moment sequence of some Schwartz function on $\R^d$ with support in $K$. We obtain a similar result for Gelfand--Shilov spaces. Several illustrative examples are discussed. Our work is inspired by a recent result of Schm\"udgen [Expositiones Math. 43 (2025), 125657], who addressed the analogous problem for Radon measures.

 \end{abstract}

\maketitle

\section{Introduction}
In his seminal work \cite{Stieltjes}, Stieltjes  characterized the non-negative sequences $(c_j)_{j\in \N}$ that are the moment sequence of some non-negative Radon measure $\mu$ on $[0,\infty)$, that is,
\begin{equation}
\label{Boass}
\int_0^\infty x^j d\mu(x) = c_j, \qquad \forall j \in \N.
\end{equation}
Since then,  the theory of moment problems has been extensively developed in various directions; see the monograph \cite{Schmudgenbook} for more information.

Boas \cite{Boas} initiated the study of  moment problems for signed measures. He proved that for \emph{every} real sequence $(c_j)_{j\in \N}$ there exists a signed Radon measure $\mu$ on $[0,\infty)$ such that \eqref{Boass} holds. P\'olya improved this result by showing that, given any unbounded set $K \subseteq [0,\infty)$, for every real sequence $(c_j)_{j\in \N}$ there exists a signed Radon measure $\mu$  on $[0,\infty)$ with $\operatorname{supp} \mu \subseteq K$ such that \eqref{Boass} holds. A multidimensional version of Boas’ result was later established by Sherman \cite{Sherman}.

Moment problems in the setting of the Schwartz space of rapidly decreasing smooth functions were first considered by  Dur\'an \cite{Duran}. He showed  that  for every real sequence $(c_j)_{j\in \N}$ there exists a Schwartz function $f \in \mathcal{S}(\R)$ with $\operatorname{supp} f \subseteq [0,\infty)$ such that 
\begin{equation}
\label{Durann}
\int_0^\infty x^j f(x)d x = c_j, \qquad \forall j \in \N.
\end{equation}
Estrada \cite{Estrada} later obtained multidimensional and vector-valued versions of Dur\'an's result.  See \cite{DebrouwereNeyt,diDio,Duran2,D-E,EV,Popov} for  related works.

J.~Chung, S.-Y.~Chung, and Kim  \cite{C-C-K} further sharpened Dur\'an's result by extending it to  Gelfand-Shilov spaces \cite{G-S}.  For $\sigma >0$ we define $\mathcal{S}^{\sigma}(\R^d)$ as the space of all $f \in C^\infty(\R^d)$ such that, for some $h >0$,
$$
\sup_{\alpha \in \N^d} \sup_{x \in \R^d} \frac{|f^{(\alpha)}(x)|}{h^{|\alpha|}|\alpha|!^\sigma}(1+|x|)^n < \infty ,\qquad \forall n \in \N.
$$
 In \cite{C-C-K} it is shown that,  given any $\sigma >1$,  for every real sequence $(c_j)_{j\in \N}$ there  exists  $f \in \mathcal{S}^\sigma(\R)$ with $\operatorname{supp} f \subseteq [0,\infty)$ such that \eqref{Durann} holds\footnote{The restriction $\sigma > 1$ is necessary: If $\sigma \leq 1$, every $f \in \mathcal{S}^\sigma(\mathbb{R})$ that is zero on some non-empty open set of $\R$ must vanish identically; this follows from the Denjoy-Carleman theorem but also more directly from the fact every $f \in \mathcal{S}^\sigma(\mathbb{R})$  can be holomorphically extended to a complex strip $\mathbb{R} + i(-\lambda,\lambda)$ for some $\lambda > 0$.}.
An extension of this result for Gelfand-Shilov spaces defined via general weight sequences was later established by Lastra and Sanz \cite{L-S09}. See \cite{C-K-Y,D-J-S,D-SolStieltjesMomProbGSSp,L-S-LinContOpStieltjesMomProbGSSp} for related works.
 
 Recently, Schm\"udgen \cite{Schmudgen} considered the following question: \emph{for which closed sets $K \subseteq \R^d$ does it hold that for every real multi-sequence $(c_\alpha)_{\alpha \in \N^d}$ there exists a signed Radon measure $\mu$ on $\R^d$ with $\operatorname{supp} \mu \subseteq K$ such that 
\begin{equation}
\label{Boass-2}
\int_{\R^d} x^\alpha d\mu(x) = c_\alpha, \qquad \forall \alpha \in \N^d?
\end{equation}}
To formulate this precisely, we define \( \mathcal{M}_{\operatorname{pol}}(\mathbb{R}^d) \) as the space of all signed Radon measures \( \mu \) on \( \mathbb{R}^d \) such that
\[
\int_{\mathbb{R}^d} (1 + |x|)^n \, d|\mu|(x) < \infty, \qquad \forall n \in \mathbb{N},
\]
where \( |\mu| \) denotes the total variation of \( \mu \).  Given a closed set \( K \subseteq \mathbb{R}^d \), we say that the \emph{unrestricted \( K \)-moment problem is solvable in \( \mathcal{M}_{\operatorname{pol}}(\mathbb{R}^d) \)} if for every real multi-sequence \( (c_\alpha)_{\alpha \in \mathbb{N}^d} \) there exists  \( \mu \in \mathcal{M}_{\operatorname{pol}}(\mathbb{R}^d) \) with \( \operatorname{supp} \mu \subseteq K \) such that \eqref{Boass-2} holds. Recall that a closed set $K \subseteq \R^d$ is  \emph{Zariski dense} if for every polynomial $P \in \R[x_1, \ldots,x_d]$ it holds that $P\equiv 0$ if $P(x) = 0$ for all $x \in K$.
Schm\"udgen showed the following result, which can be viewed as a multidimensional version of P\'olya's theorem \cite{Polya}\footnote{Schm\"udgen actually showed a more general result  for  finitely generated commutative unital real algebras $A$ \cite[Theorem 2]{Schmudgen}. Theorem  \ref{main-schmudgen} corresponds to the special case when $A$ is the polynomial algebra $\R[x_1,\ldots,x_d]$.}:
\begin{theorem} \cite[Theorem 5]{Schmudgen} \label{main-schmudgen} Let $K \subseteq \R^d$ be closed. The following statements are equivalent:
\begin{enumerate}[(1)]
\item  The unrestricted $K$-moment problem is solvable in $\mathcal{M}_{\operatorname{pol}}(\R^d)$.
\item $K$ is Zariski dense and for all $n \in \N$ the space
$$
\mathcal{N}_n(K) = \{ P \in \R[x_1,\ldots, x_d] \mid \sup_{x \in K} \frac{|P(x)|}{(1+|x|)^n} < \infty \}
$$
is finite-dimensional.
\end{enumerate} 
\end{theorem} 

In this article, we investigate the analogue of Schm\"udgen's question for the Schwartz space $\mathcal{S}(\R^d)$ and the Gelfand-Shilov spaces $\mathcal{S^\sigma}(\R^d)$ ($\sigma >1$). More precisely, we ask:
 \emph{for which closed sets $K \subseteq \R^d$ does it hold that for every real multi-sequence $(c_\alpha)_{\alpha \in \N^d}$ there exists $f \in \mathcal{S}(\R^d)$ ($f  \in \mathcal{S}^{\sigma}(\R^d)$) with $\operatorname{supp} f \subseteq K$ such that 
\begin{equation}
\label{Durann-2}
\int_{\R^d} x^\alpha f(x)dx = c_\alpha, \qquad \forall \alpha \in \N^d?
\end{equation}}
If this is the  case, we say that the  \emph{unrestricted \( K \)-moment problem is solvable in $\mathcal{S}(\R^d)$ ($\mathcal{S}^\sigma(\R^d)$)}. Given a closed set \( K \subseteq \mathbb{R}^d \), we write
$$
d_K(x) = \min \{1, d(x, \partial K) \}, \qquad x \in K,
$$
where $d(x, \partial K)$ denotes the distance from $x$ to the boundary $\partial K$ of $K$. Recall that  a closed set $K \subseteq \R^d$ is said to be \emph{regular} if $K = \overline{\operatorname{int} K}$. Our two main theorems provide a complete solution to the above question:
\begin{theorem} \label{main-pol} Let $K \subseteq \R^d$ be a regular closed set\footnote{There is no loss of generality in assuming that $K$ is regular: If $L \subseteq \R^d$ is closed and $f \in C(\R^d)$ with $\operatorname{supp} f \subseteq L$, then $\operatorname{supp} f \subseteq  \overline{\operatorname{int} L}$. This implies that the unrestricted $L$-moment problem is solvable in $\mathcal{S}(\R^d)$ if and only if  the unrestricted $\overline{\operatorname{int} L}$-moment problem  is so. The same remark applies to the Gelfand-Shilov spaces $\mathcal{S}^\sigma(\R^d)$ ($\sigma >1$).}. The following statements are equivalent:
\begin{enumerate}[(1)]
\item The unrestricted $K$-moment problem is solvable in $\mathcal{S}(\R^d)$.
\item For all $k,n \in \N$ the space 
$$
\mathcal{N}_{k,n}(K) =\left \{ P \in  \R[x_1, \ldots, x_d]  \, | \, \sup_{x\in K}\frac{|P(x)|d_K(x)^k}{(1 + |x|)^n} < \infty \right\}
$$
is finite-dimensional.
\end{enumerate} 
\end{theorem} 

\begin{theorem} \label{main-GS} Let $K \subseteq \R^d$ be a  regular closed set and $\sigma >1$. The following statements are equivalent:
\begin{enumerate}[(1)]
\item The unrestricted $K$-moment problem is solvable in $\mathcal{S}^\sigma(\R^d)$.
\item There exists $\varepsilon >0$ such that for all $n \in \N$ the space 
$$
\mathcal{N}^{\sigma,\varepsilon}_{n}(K) =\left \{ P \in  \R[x_1, \ldots, x_d]  \, | \, \sup_{x\in K}\frac{|P(x)|e^{-\varepsilon\left(\frac{1}{d_K(x)}\right)^{\frac{1}{\sigma-1}}}}{(1 + |x|)^n} < \infty \right\}
$$
is finite-dimensional.
\end{enumerate} 
\end{theorem} 
Roughly speaking, the conditions in (2) in Theorems \ref{main-pol} and \ref{main-GS} express that the set $K$ is sufficiently "thick at infinity", where the exact notion of thickness depends on the regularity of the solution space $\mathcal{S}(\R^d)$ or $\mathcal{S}^\sigma(\R^d)$. We shall present necessary and sufficient geometric criteria that illustrate this principle. As a corollary of these criteria, we will show that the solvability of the unrestricted $K$-moment problem in $\mathcal{S}^\sigma(\R^d)$ genuinely depends on the index $\sigma$:
\begin{theorem}\label{KGSGS}
Let $\sigma > \tau >1$. There exists a  regular closed set  $K \subseteq \R^d$ such that the unrestricted $K$-moment problem is solvable in $\mathcal{S}^\sigma(\R^d)$ but not in $\mathcal{S}^\tau(\R^d)$. 
\end{theorem}
We will in fact show extensions of Theorems \ref{main-GS} and \ref{KGSGS} for Gelfand-Shilov spaces defined via general weight sequences.

Our proofs of Theorems \ref{main-GS} and \ref{KGSGS} rely on a classical result of Eidelheit \cite{Eidelheit}, which characterizes the solvability of countably infinite systems of continuous  linear equations in a Fr\'echet space. Eidelheit's theorem is tailor-made to study various surjectivity problems, e.g., it yields concise proofs of Borel's theorem and Weierstrass' interpolation theorem. 
While very natural, to the best of our knowlegde, Eidelheit's theorem has not yet been applied in the study of unrestricted moment problems (although it was implicitly used in \cite{EV}). As an illustration, we provide a short alternative proof of Theorem \ref{main-schmudgen} based on it.  For further developments and applications of Eidelheit's theorem, see \cite{D-N-ExtLBSpSurjTensMap,DomanskiVogt,V-Kernel,V-Mit}.

This paper is organized as follows. In Section \ref{sect-eid}, we present Eidelheit's theorem and use it to obtain a new proof of Theorem \ref{main-schmudgen}. The proof of Theorem \ref{main-pol} is given in Section \ref{sect-Schwartz}. In the preliminary Section \ref{sect-GS-prelim}, we introduce weight sequences and the associated Gelfand–Shilov spaces that will be used throughout the paper. Section \ref{sect-GS} is dedicated to the proof of  the extension of Theorem \ref{main-GS} to general Gelfand-Shilov spaces. Finally, in Section \ref{sect-examples}, we give some applications of our main results, discuss several examples, and prove   the extension of  Theorem \ref{KGSGS} to general Gelfand-Shilov spaces.

\section{Eidelheit's theorem}\label{sect-eid}
In this section, we state Eidelheit's theorem and use it to give a short alternative proof of Theorem \ref{main-schmudgen}.

Let $E$ be a locally convex space. We write $E'$ for the topological dual of $E$. Given a subspace $F \subseteq E'$ and a continuous seminorm  $p$ on $E$, we define $F_p$ as the space consisting of all elements in $F$  that are continuous with respect to $p$. Eidelheit's theorem \cite{Eidelheit}, may now be formulated as follows:
\begin{theorem}[Eidelheit's theorem]  \cite[Theorem 26.27]{MV}\label{Eid}
Let $E$ be a Fr\'echet space and $\mathcal{P}$ be  a fundamental family  of  continuous seminorms on $E$. Let $(e'_j)_{j \in J} \subseteq E'$ be a countable family. The following statements are equivalent:
\begin{enumerate}[(1)]
\item  For every  $(c_j)_{j \in J} \in \R^{J}$ there exists  $e \in  E$ such that 
$$
\langle e'_j, e \rangle = c_j, \qquad \forall j \in J.
$$
\item The family $ (e'_{j})_{j \in J}$ is linearly independent and for all $p \in \mathcal{P}$ the space
$$(\operatorname{span} \{e'_{j} \, | \, j \in J \})_p$$ 
is finite-dimensional.
\end{enumerate}
\end{theorem}
We now use Eidelheit's theorem to show Theorem \ref{main-schmudgen}.  For $n \in \N$ we define the norm
$$
\|\mu\|_n = \int_{\R^d} (1+|x|)^n d|\mu|(x), \qquad \mu \in \mathcal{M}_{\operatorname{pol}}(\R^d).
$$
We endow  $\mathcal{M}_{\operatorname{pol}}(\R^d)$ with the locally convex topology generated by the system of norms $\{ \| \, \cdot \, \|_n \mid n \in \N\}$. Note that $\mathcal{M}_{\operatorname{pol}}(\R^d)$ is a Fr\'echet space. For $K \subseteq \R^d$ closed we define $\mathcal{M}_{\operatorname{pol},K} = \{ \mu \in \mathcal{M}_{\operatorname{pol}}(\R^d) \mid \operatorname{supp} \mu \subseteq K\}$ and endow this space with the subspace topology induced by  $\mathcal{M}_{\operatorname{pol}}(\R^d)$. Since $\mathcal{M}_{\operatorname{pol},K}$ is a closed subspace of  $\mathcal{M}_{\operatorname{pol}}(\R^d)$, it is also a Fr\'echet space.
We need the following simple lemma.
\begin{lemma}\label{lemma11} Let  $K \subseteq \R^d$  be closed  and $P \in \R[x_1,\ldots x_d]$.  
\begin{enumerate}[(1)]
\item $P(x) = 0$ for all $x \in K$ if and only if 
$$
\int_{\R^d} f(x) d\mu(x) = 0, \qquad \forall \mu \in \mathcal{M}_{\operatorname{pol},K}.
$$
\item Let $n \in \N$. $P \in \mathcal{N}_n(K)$ if and only if there is $C >0$ such that 
$$
\left| \int_{\R^d} P(x) d\mu(x) \right| \leq C\| \mu\|_n, \qquad \forall \mu \in \mathcal{M}_{\operatorname{pol},K}.
$$
\end{enumerate}
\end{lemma}
\begin{proof}
For $x \in K$ we write $\delta_x \in  \mathcal{M}_{\operatorname{pol},K}$ for the Dirac delta measure concentrated at $x$. \\
(1) $\Rightarrow$: Obvious. \\
 $\Leftarrow$: For all $x \in K$ it holds that
$$
P(x) = \int_{\R^d} P(y) d \delta_x(y) = 0. 
$$
(2) $\Rightarrow$: Obvious. \\
$\Leftarrow$: For all $x \in K$ it holds that
$$
|P(x)| = \left |\int_{\R^d} P(y) d \delta_x(y)  \right| \leq C \|  \delta_x \|_n = C\int_X (1+|y|)^n d \delta_x(y) = C(1+|x|)^n.
$$
\end{proof}

\begin{proof}[Alternative proof of Theorem \ref{main-schmudgen}]
For $P \in \R[x_1,\ldots, x_d]$ we define $ e'_P \in (\mathcal{M}_{\operatorname{pol},K})'$ via
$$
\langle e'_P, \mu \rangle = \int_{\R^d} P(x) d\mu(x), \qquad \mu \in \mathcal{M}_{\operatorname{pol},K}.
$$
Clearly, the unrestricted $K$-moment problem is solvable in $\mathcal{M}_{\operatorname{pol}}(\R^d)$ if and only if for every  $(c_\alpha)_{\alpha \in \N^d} \in \R^{\N^d}$ there exists $\mu \in  \mathcal{M}_{\operatorname{pol},K}$  such that 
$$
\langle e'_{x^\alpha}, \mu \rangle = c_\alpha, \qquad \forall \alpha \in \N^d.
$$
Since $\mathcal{M}_{\operatorname{pol},K}$ is a Fr\'echet space, it suffices by Eidelheit's theorem to show that the following two sets of conditions are equivalent to each other: 
\begin{enumerate}
\item[(a)] $(e'_{x^\alpha})_{\alpha \in \N^d}$ is linearly independent.
\item[(b)] For all $n \in \N$  the space $(\operatorname{span} \{e'_{x^\alpha} \mid \alpha \in \N^d\})_{\| \, \cdot \, \|_n}$ is finite-dimensional.
\end{enumerate}
and
\begin{enumerate}
\item[(a')] $K$ is Zariski dense.
\item[(b')] For all $n \in \N$ the space $\mathcal{N}_n(K)$ is finite-dimensional.
\end{enumerate}
Consider the linear mapping 
$$T: \R[x_1, \ldots, x_d] \to (\mathcal{M}_{\operatorname{pol},K})', \, P \mapsto e'_P.
$$
Condition (a) holds if and only if $T$ is injective. Hence, Lemma \ref{lemma11}(1) implies  that (a) and (a') are equivalent. Next, we assume that (a) holds and show that (b) is equivalent to (b'). Since we already know that (a) and (a') are equivalent, this will finish the proof. Note that $\operatorname{span} \{e'_{x^\alpha} \mid \alpha \in \N^d\} = \{e'_{P} \mid P \in \R[x_1, \ldots, x_d]\}$. Lemma \ref{lemma11}(2) implies that $T(\mathcal{N}_n(K)) =
\{e'_P \mid P \in \R[x_1, \ldots, x_d]\}_{\| \, \cdot \, \|_n}$  for all $n \in \N$. This yields that (b) and (b') are equivalent.
\end{proof}

\section{The unrestricted $K$-moment problem in $\mathcal{S}(\R^d)$ }\label{sect-Schwartz}
This section is dedicated to the proof of Theorem \ref{main-pol}. We adopt the same strategy as in the alternative proof of Theorem \ref{main-schmudgen} from the previous section, relying on Eidelheit's theorem. The main challenge will be to establish appropriate analogues of the statements in Lemma \ref{lemma11} within the present setting.

For $k,n \in \N$ we define the norm
$$
\| f \|_{k,n} = \max_{|\alpha| \leq k} \sup_{x \in \R^d}  |f^{(\alpha)}(x)|(1+|x|)^n, \qquad f \in  \mathcal{S}(\R^d).
$$
We endow  $\mathcal{S}(\R^d)$ with the locally convex topology induced by the system of norms $\{ \| \, \cdot \, \|_{k,n} \mid k,n \in \N \}$. Note that $\mathcal{S}(\R^d)$ is a Fr\'echet space. For $K \subseteq \R^d$ closed we write $\mathcal{S}_{K} = \{ f \in \mathcal{S}(\R^d) \mid  \operatorname{supp} f \subseteq K\}$ and endow this space with the subspace topology induced by $\mathcal{S}(\R^d)$. Since $\mathcal{S}_{K}$  is a  closed subspace of $\mathcal{S}(\R^d)$, it is also a Fr\'echet space. If $K$ is compact,  we write $\mathcal{D}_{K} = \mathcal{S}_{K}$.

We start with the analogue of  Lemma \ref{lemma11}(1).
\begin{proposition}\label{lemma1} Let $K \subseteq \R^d$ be a non-empty regular closed set and $P \in \R[x_1, \ldots, x_d]$. The following statements are equivalent:
\begin{enumerate}[(1)]
\item $P \equiv 0$.
\item $P(x) = 0$ for all $x \in K$.
\item It holds that 
$$
\int_{\R^d} P(x) f(x) dx = 0, \qquad \forall f \in \mathcal{S}_{K}.
$$
\end{enumerate}
\end{proposition}
\begin{proof}
The implications $(1) \Rightarrow (2) \Rightarrow (3)$ are trivial. We now show $(3) \Rightarrow (1)$. Since $P \equiv 0$ if $P$ vanishes on some non-empty open subset of $\R^d$, it suffices to show that $P(x) = 0$ for all $x \in \operatorname{int} K$. Pick $\varphi \in \mathcal{D}_{\overline{B}(0,1)}$ with $\int_{\R^d}\varphi(y)dy = 1$. Set
$\varphi_{x,\varepsilon}(y) = \varepsilon^{-d}\varphi((x-y)/\varepsilon)$ for $\varepsilon >0$ and $x \in \R^d$.   Let $x \in \operatorname{int} K$ be arbitrary. For all $0 < \varepsilon < d(x,\partial K)$ it holds that $\varphi_{x,\varepsilon} \in  \mathcal{S}_{K}$. Consequently,
$$
P(x) = \lim_{\varepsilon \to 0^+} \int_{\R^d} P(y) \varphi_{x,\varepsilon} (y)  dy = 0.
$$
\end{proof}

Next, we show the analogue of  Lemma \ref{lemma11}(2), which we divide into two separate statements. 
%

\begin{proposition}\label{lemma2}  Let $K \subseteq \R^d$ be a  regular closed set and $k,n \in \N$. For all $P \in \mathcal{N}_{k,n}(K)$ there is $C >0$ such that 
$$
\left| \int_{\R^d} P(x) f(x) dx \right| \leq C \| f \|_{k,n+d+1}, \qquad \forall f \in  \mathcal{S}_{K}.
$$
\end{proposition}
\begin{proof}
Since  $P \in  \mathcal{N}_{k,n}(K)$, it holds that
$$
C_0 = \sup_{x\in K}\frac{|P(x)|d_K(x)^{k}}{(1 + |x|)^n} < \infty.
$$
Set $K_0 = \{ x \in \operatorname{int} K \, | \, d(x,\partial K) \leq 1\}$ and $C_1 = \int_{\R^d}(1+|x|)^{-d-1} dx$. For all $f \in  \mathcal{S}_{K}$
 we find that
\begin{align*}
\left| \int_{\R^d} P(x) f(x) dx \right| &\leq  \int_{K_0} |P(x)| |f(x)| dx  +  \int_{K \backslash K_0} |P(x)| |f(x)| dx  \\
&\leq C_0C_1 \left ( \sup_{x\in K_0} \frac{ |f(x)|(1 + |x|)^{n+d+1}}{d(x,\partial K)^{k}} +  \sup_{x\in K \backslash K_0} |f(x)|(1 + |x|)^{n+d+1} \right) \\
&\leq C_0C_1 \left ( \sup_{x\in K_0} \frac{ |f(x)|(1 + |x|)^{n+d+1}}{d(x,\partial K)^{k}} + \| f \|_{k,n+d+1}  \right).
\end{align*}
Set $m= n+d+1$. By the above inequality, it suffices to show that there is $C_2>0$ such that
$$
\sup_{x\in K_0} \frac{ |f(x)|(1 + |x|)^{m}}{d(x,\partial K)^{k}}\leq C_2\| f \|_{k,m},  \qquad \forall  f \in  \mathcal{S}_{K}.
$$
Let $ f \in  \mathcal{S}_{K}$ and  $x \in K_0$ be arbitrary. Pick $\xi \in \partial K$ such that $|x-\xi| = d(x,\partial K)$. Note that $f^{(\alpha)}(\xi) =0$ for all $\alpha \in \N^d$. Hence, Taylor's formula implies that there exists $\gamma \in \{ tx + (1-t)\xi \mid t \in [0,1]\}$ such that
$$
f(x) = \sum_{|\alpha| = k} \frac{f^{(\alpha)}(\gamma)}{\alpha!} (x-\xi)^\alpha.
$$
As $x \in K_0$, it holds that $|x-\gamma| \leq |x-\xi| = d(x,\partial K) \leq 1$. Therefore,
$$
(1+|\gamma|)^{-m} \leq (1+|\gamma-x|)^m(1+|x|)^{-m} \leq  2^m(1+|x|)^{-m}.
$$ 
Set $C_3 = \sum_{|\alpha| = k} \frac{1}{\alpha !}$. We obtain that 
$$
|f(x)| \leq   |x-\xi|^k \sum_{|\alpha| = k} \frac{|f^{(\alpha)}(\gamma)|}{\alpha!}  \leq C_3 \| f\|_{k,m} \frac{d(x, \partial K)^k}{(1+|\gamma|)^m} \leq 2^mC_3 \| f\|_{k,m} \frac{d(x, \partial K)^k}{(1+|x|)^m}.
$$
\end{proof}

\begin{proposition}\label{lemma3}  Let $K \subseteq \R^d$ be a regular closed set and $k,n \in \N$. Let $P \in \R[x_1, \ldots, x_d]$ be such that, for some $C >0$,
\begin{equation}
 \label{ct}
\left| \int_{\R^d} P(x) f(x) dx \right| \leq C \| f \|_{k,n}, \qquad \forall f \in  \mathcal{S}_{K}.
\end{equation}
Then, $P \in \mathcal{N}_{k+d,n}(K)$. 
\end{proposition}

We need an auxiliary result for the proof of Proposition \ref{lemma3}.  For $m \in \N$ we denote by $\R_m[x_1, \ldots, x_d]$ the space consisting of all $P \in \R[x_1, \ldots, x_d]$ with $\operatorname{deg} P \leq m$.  

\begin{lemma}\label{lemma131} Let $k,m \in \N$. There is $C >0$ such that for all $P \in \R_m[x_1, \ldots, x_d]$, $x \in \R^d$, and  $0 < r \leq 1$,
$$
 |P(x)| \leq \frac{C}{r^{k+d}} \sup \left \{ \left | \int_{\R^d} P(y)\varphi(y) dy \right | \mid \varphi \in \mathcal{D}_{\overline{B}(x,r)}, \| \varphi\|_{k,0} \leq 1 \right \}.
$$
\end{lemma}
\begin{proof}
We write
$$
\|P\|_\infty = \sup_{y \in \overline{B}(0,1)} |P(y)|, \qquad P \in \R_m[x_1, \ldots, x_d],
$$
and
$$
\|P\| = \sup \left \{ \left | \int_{\R^d} P(y) \psi(y) dy \right | \mid \psi \in \mathcal{D}_{\overline{B}(0,1)}, \| \psi \|_{k,0} \leq 1 \right \}, \qquad P \in \R_m[x_1, \ldots, x_d].
$$
By Lemma \ref{lemma11} (with $K = \overline{B}(0,1)$), $\|\, \cdot \,\|_\infty$ and $\|\, \cdot \,\|$ are both norms on  $\R_m[x_1, \ldots, x_d]$. Since the space $\R_m[x_1, \ldots, x_d]$ is finite-dimensional, $\|\, \cdot \, \|_\infty$ and $\|\, \cdot \, \|$ are equivalent. Hence, there is $C >0$ such that 
$$
\|P\|_\infty \leq C \|P\|, \qquad \forall P \in \R_m[x_1, \ldots, x_d].
$$
Let $P \in \R_m[x_1, \ldots, x_d]$, $x \in \R^d$,  and $0 < r \leq 1$  be arbitrary.  Define $P_{x,r}(y) = P(x +ry)$ and note that $P_{x,r} \in \R_m[x_1, \ldots, x_d]$. 
For  $\psi \in \mathcal{D}_{\overline{B}(0,1)}$ we set $\psi_{x,r}(y) = \psi((y-x)/r)$. Then, $\psi_{x,r} \in  \mathcal{D}_{\overline{B}(x,r)}$ and
$$
 \| \psi_{x,r}\|_{k,0}  \leq \frac{1}{r^{k}} \| \psi\|_{k,0}.
$$
Therefore, we obtain that
\begin{align*}
&  |P(x)| \leq  \| P_{x,r}\|_\infty  \leq C \|P_{x,r}\| \\
&=  C \sup \left \{ \left | \int_{\R^d} P_{x,r}(y) \psi(y) dy \right | \mid \psi \in \mathcal{D}_{\overline{B}(0,1)},  \| \psi\|_{k,0} \leq 1 \right \} \\
&=  \frac{C}{r^d}\sup \left \{ \left | \int_{\R^d} P(y) \psi_{x,r}(y) dy \right | \mid   \psi \in \mathcal{D}_{\overline{B}(0,1)},  \| \psi\|_{k,0} \leq 1 \right \} \\
&\leq \frac{C}{r^{k+d}} \sup \left \{ \left | \int_{\R^d} P(y) \varphi(y) dy \right | \mid \varphi \in \mathcal{D}_{\overline{B}(x,r)},   \| \varphi\|_{k,0} \leq 1  \right \}.
\end{align*}

\end{proof}

\begin{proof}[Proof of Proposition \ref{lemma3}]
We need to show that there is $C >0$ such that
$$
|P(x)| \leq \frac{C(1+|x|)^n}{d_K(x)^{k+d}}, \qquad \forall x \in \operatorname{int} K.
$$
 Lemma \ref{lemma131} implies that there is $C_1>0$ such that for all $x \in \operatorname{int} K$ (choose $r = d_K(x)$)
$$
|P(x)|  \leq \frac{C_1}{d_K(x)^{k+d}} \sup \left \{ \left | \int_{\R^d} P(y)\varphi(y) dy \right | \mid \varphi \in \mathcal{D}_{\overline{B}(x,d_K(x))}, \| \varphi\|_{k,0} \leq 1 \right \}.
$$
Let $x \in \operatorname{int} K$ and $\varphi \in \mathcal{D}_{ \overline{B}(x,d_K(x))}$ be arbitrary. Note that $\varphi \in \mathcal{S}_K$.
As $d_K(x) \leq 1$, it holds that 
$$
(1+|y|)^{n} \leq (1+|y-x|)^n(1+|x|)^n \leq  2^n(1+|x|)^{n}, \qquad \forall y \in \overline{B}(x,d_K(x)).
$$ 
We obtain that 
$$
\|\varphi\|_{k,n} = \max_{|\alpha| \leq k} \sup_{y \in \overline{B}(x,d_K(x))} |\varphi^{(\alpha)}(y)| (1+|y|)^n \leq 2^n(1+|x|)^n \|\varphi\|_{k,0}.
$$
By  \eqref{ct}, we find that
\begin{align*}
|P(x)| &\leq \frac{C_1}{d_K(x)^{k+d}} \sup \left \{ \left | \int_{\R^d} P(y)\varphi(y) dy \right | \mid \varphi \in \mathcal{D}_{ \overline{B}(x,d_K(x))}, \| \varphi\|_{k,0} \leq 1 \right \} \\
&\leq \frac{CC_1}{d_K(x)^{k+d}} \sup \{ \|\varphi\|_{k,n}  \mid \varphi \in \mathcal{D}_{ \overline{B}(x,d_K(x))}, \| \varphi\|_{k,0} \leq 1  \} \\
&\leq \frac{2^nCC_1(1+|x|)^n}{d_K(x)^{k+d}}. 
\end{align*}
\end{proof}
We now show  Theorem \ref{main-pol}. The  proof is very similar to the one of Theorem \ref{main-schmudgen}.

\begin{proof}[Proof of Theorem \ref{main-pol}] We may assume that $K$ is non-empty.
For $P \in \R[x_1,\ldots, x_d]$ we define $ e'_P \in (\mathcal{S}_{K})'$ via
$$
\langle e'_P, f \rangle = \int_{\R^d} P(x) f(x)dx, \qquad f \in \mathcal{S}_{K}.
$$
Clearly, the unrestricted $K$-moment problem is solvable in $\mathcal{S}(\R^d)$ if and only if for every  $(c_\alpha)_{\alpha \in \N^d} \in \R^{\N^d}$ there exists $f \in \mathcal{S}_{K}$  such that 
$$
\langle e'_{x^\alpha}, f \rangle = c_\alpha, \qquad \forall \alpha \in \N^d.
$$
Consider the linear mapping 
$$T: \R[x_1, \ldots, x_d] \to (\mathcal{S}_{K})', \, P \mapsto e'_P.
$$
Proposition \ref{lemma1} implies that $T$ is injective, which is equivalent to the fact that the  family $(e'_{x^\alpha})_{\alpha \in \N^d}$ is linearly independent. Since $\mathcal{S}_K$ is a Fr\'echet space, it suffices  by  Eidelheit's theorem to show that  $(\operatorname{span} \{e'_{x^\alpha} \mid \alpha \in \N^d\})_{\| \, \cdot \, \|_{k,n}}$ is finite-dimensional for all $k,n \in \N$ if and only if  $\mathcal{N}_{k,n}(K)$ is finite-dimensional for all $k,n \in \N$. Note that $\operatorname{span} \{e'_{x^\alpha} \mid \alpha \in \N^d\} = \{e'_{P} \mid P \in \R[x_1, \ldots, x_d]\}$. Propositions  \ref{lemma2} and \ref{lemma3} imply that for all $k,n \in \N$
$$
T(\mathcal{N}_{k,n}(K)) \subseteq \{e'_{P} \mid P \in \R[x_1, \ldots, x_d]\}_{\| \, \cdot \, \|_{k,n+d+1}}
$$
and
$$
 T^{-1}(\{e'_{P} \mid P \in \R[x_1, \ldots, x_d]\}_{\| \, \cdot \, \|_{k,n}}) \subseteq \mathcal{N}_{k+d,n}(K).
$$
As $T$ is injective, this implies the result.
\end{proof}

\section{Weight sequences and Gelfand-Shilov spaces}\label{sect-GS-prelim}
In this preliminary section, we introduce weight sequences and establish several technical properties of them; see  \cite{Komatsu, Rainer} for more information. We also define the Gelfand–Shilov spaces associated with these sequences that we will be interested in.

By a \emph{weight sequence} we mean a positive sequence $M = (M_p)_{p \in \N}$ satisfying
\begin{itemize}
\item $M_0 =1 \leq M_1$.
\item $M$ is log-convex, i.e., $M^2_{p} \leq M_{p-1}M_{p+1}$ for all $p \geq 1$.
\item $M$ is  non-quasianalytic, i.e., 
$$
\sum_{p=1}^\infty \frac{M_{p-1}}{M_p} < \infty.
$$ 
\end{itemize}
We consider the following additional conditions on a weight sequence $M$:
\begin{itemize}
	\item[$(M.2)$]  $\displaystyle \exists C> 0 \, \forall p,q \in \N  : M_{p+q} \leq C^{p+q} M_{p} M_{q}$.
	\item[$(M.3)$] $\displaystyle \exists C> 0 \, \forall p \in \N  :   \sum_{q = p+1}^\infty \frac{M_{q-1}}{M_q} \leq Cp \frac{M_{p}}{M_{p+1}}$.
\end{itemize}
Condition $(M.2)$ is often referred to as the moderate growth condition and $(M.3)$ as the strong non-quasianalyticity condition.
\begin{example}
The weight sequence $M_\sigma = (p!^{\sigma})_{p \in \N}$, $\sigma >1$, is called the \emph{Gevrey sequence} with index $\sigma$. It satisfies $(M.2)$ and $(M.3)$.
\end{example}
Let $M$ and $N$ be two weight sequences. We write $N \subset M$ if there are $C,h >0$ such that $N_p \leq Ch^p M_p$ for all $p \in \N$. The stronger relation $N \prec M$ means that the latter inequality is valid for every $h >0$ and suitable $C >0$. We write $N \asymp M$ if both $N \subset M$ and $M \subset N$.


Let $M$  be a weight sequence.  We define
$$
\nu_M(t) = \inf_{p \in \N} \frac{t^pM_p}{p!}, \qquad t >0,
$$
and $\nu_M(0) = 0$. The function $\nu_M$ is increasing, $ 0 \leq \nu_M \leq 1$, and  for all $a >0$ there is $C >0$ such that
\begin{equation}\label{sublog}
\nu_M(t) \leq C t^a, \qquad \forall t \geq 0.
\end{equation}
As $M$ is non-quasianaltyic, it holds that $(M_p/p!)^{1/p} \to \infty$ as $p \to \infty$ \cite[Lemma 4.1]{Komatsu}. Consequently, $\nu_M(0) >0$ for all $t >0$ and $\nu_M$ is continuous on $[0,\infty)$. The function $\nu_M$ will play an essential role throughout the rest of this article.
\begin{example}\label{weightG}
Consider the Gevrey sequence $M_\sigma$ of index $\sigma >1$. There are $C_0,C_1,H_0,H_1 >0$ such that 
$$
C_0e^{-H_0\left(\frac{1}{t}\right)^{\frac{1}{\sigma-1}}} \leq \nu_{M_\sigma}(t) \leq C_1e^{-H_1\left(\frac{1}{t}\right)^{\frac{1}{\sigma-1}}}, \qquad \forall t > 0.
$$
\end{example}
\begin{lemma}\label{propnu} Let $M$ be a weight sequence.
\begin{enumerate}[(1)]
\item Suppose that $M$ satisfies $(M.2)$. For all $a >0$ there is $C >0$ such that 
$$
\nu_M(t) \leq \nu_M(Ct)^a, \qquad \forall t \geq 0. 
$$
\item Suppose that $M$ satisfies $(M.2)$. There is $C_0 >0$ such that for all $a > 0$ there is $C_1 >0$
$$
\nu_M(t) \leq C_1t^a \nu_M(C_0t), \qquad \forall t \geq 0. 
$$
\item Suppose that $M$ satisfies $(M.3)$. For all $a >0$ there are  $C_0,C_1>0$  such that 
$$
\nu_M(at)^{C_0} \leq C_1\nu_M(t), \qquad \forall t \geq 0. 
$$
\end{enumerate}
\end{lemma}
\begin{proof}
(1) Since $\nu_M(t) \leq 1$ for all $t \geq 0$,  it suffices to consider the cases $a = 2^n$, $n \in \N$. The case $a = 2$  is shown in \cite[Lemma 7.2]{Rainer} and the general result then follows by induction. \\
(2) This follows from \eqref{sublog} and part (1) (with $a =2$). \\
(3) Since $\nu_M$ is increasing,  it suffices to consider the cases $a = 2^n$, $n \in \N$. We will show the case $a = 2$, the general result then follows by induction.
We write $M^* = (M_p/p!)_{p \in \N}$ and define the associated function of \( M^* \) (cf.~\cite{Komatsu}) by
\[
\omega_{M^*}(\rho) = \sup_{p \in \mathbb{N}} \log \frac{\rho^p}{M^*_p}, \qquad \rho \geq 0.
\]
Note that
\begin{equation}
\label{assfunction}
\nu_M(t) = e^{-\omega_{M^*}\left ( \frac{1}{t} \right)}, \qquad t >0.
\end{equation}
We will establish the result by combining several results from \cite{JJ} about the growth indices \( \gamma(M) \), \( \gamma(M^*) \), and \( \gamma(\omega_{M^*})\); we refer the reader to \cite{JJ} for their definition. Since $M$ satisfies $(M.3)$ (denoted $(\gamma_1)$ in \cite{JJ}),  it follows from \cite[Corollary 3.12(ii)]{JJ} that $\gamma(M) >1$.  By \cite[Corollary 4.11]{JJ}, we obtain that 
$\gamma(M^*) >0$. In view of \cite[Corollary 4.6(1)]{JJ}, it holds that $\gamma(\omega_{M^*}) \geq \gamma(M^*) >0$. Hence, by \cite[Corollary 2.14]{JJ},  $\omega_{M^*}$ satisfies the condition $(\omega_1)$, i.e., $\omega_{M^*}(2\rho) = O(\omega_{M^*}(\rho))$ as $\rho \to \infty$. The result now follows from \eqref{assfunction}. 
\end{proof}
Finally, we introduce the class of Gelfand-Shilov spaces we will work with. Given a weight sequence $M$, we write $M_\alpha = M_{|\alpha|}$ for $\alpha \in \N^d$. For $h>0$ we define $\mathcal{S}^{M,h}(\R^d)$ as the space of all  $f \in C^\infty(\R^d)$ such that for all $n \in \N$
\begin{equation}
\label{norm3}
\| f \|^{M,h}_{n} = \sup_{\alpha \in \N^d} \sup_{x \in \R^d} \frac{|f^{(\alpha)}(x)|}{h^{|\alpha|}M_\alpha}(1+|x|)^n < \infty.
\end{equation}
We set 
$$
\mathcal{S}^{\{M\}}(\R^d) = \bigcup_{h >0} \mathcal{S}^{M,h}(\R^d).
$$
 \begin{example}
 The spaces $\mathcal{S}^{\sigma}(\R^d)$, $\sigma >0$, considered in the introduction coincide with $\mathcal{S}^{\{M_\sigma\}}(\R^d)$, where $M_\sigma$ is the Gevrey sequence of index $\sigma$.
 \end{example}

\section{The unrestricted $K$-moment problem in Gelfand-Shilov spaces}\label{sect-GS}
We study the unrestricted $K$-moment problem in the Gelfand-Shilov spaces $\mathcal{S}^{\{M\}}(\R^d)$ in this section.

 Let $K \subseteq \R^d$ be closed and $M$ be a weight sequence. We say that the \emph{unrestricted $K$-moment problem is solvable in $\mathcal{S}^{\{M\}}(\R^d)$} if for every  $(c_\alpha)_{\alpha \in \N^d} \in \R^{\N^d}$ there is $f \in \mathcal{S}^{\{M\}}(\R^d)$ with $\operatorname{supp} f \subseteq K$ such that 
\begin{equation}
\label{moment-gs}
\int_{\R^d} x^\alpha f(x) dx = c_\alpha, \qquad \forall \alpha \in \N^d.
\end{equation}
Similarly, given $h >0$, we say that the \emph{unrestricted $K$-moment problem is solvable in $\mathcal{S}^{M,h}(\R^d)$} if for every $(c_\alpha)_{\alpha \in \N^d} \in \R^{\N^d}$ there is  $f \in \mathcal{S}^{M,h}(\R^d)$ with $\operatorname{supp} f \subseteq K$ such that \eqref{moment-gs} holds. The main result of this section is:
\begin{theorem} \label{main-GSgen} Let $K \subseteq \R^d$ be a  regular closed set and $M$ be a weight sequence. Consider the following statements:
\begin{enumerate}[(1)]
\item  The unrestricted $K$-moment problem is solvable in $\mathcal{S}^{\{M\}}(\R^d)$.
\item There is $h >0$ such that the unrestricted $K$-moment problem is solvable in $\mathcal{S}^{M,h}(\R^d)$.
\item There is $h >0$ such that for every $n \in \N$ the space  
$$
\mathcal{N}^{M,h}_{n}(K) =\left \{ P \in  \R[x_1, \ldots, x_d]  \, | \, \sup_{x\in K}\frac{|P(x)|\nu_M(hd_K(x))}{(1 + |x|)^n} < \infty \right\}
$$
 is finite-dimensional. 
\end{enumerate} 
Then, $(1) \Leftrightarrow (2) \Rightarrow (3)$. If $M$ additionally  satisfies $(M.2)$ and $(M.3)$, then also $(3) \Rightarrow (1)$.
\end{theorem} 

\begin{example}\label{exampleGSfind}
Consider the Gevrey sequence $M_\sigma$ of index $\sigma >1$. Example \ref{weightG} implies that there are $H_0,H_1 >0$ such that  for all $h >0$ and $n \in \N$
$$
\mathcal{N}^{\sigma,\frac{H_1}{h^{1/(\sigma-1)}}}_{n}(K) \subseteq \mathcal{N}^{M_\sigma,h}_{n}(K) \subseteq \mathcal{N}^{\sigma,\frac{H_0}{h^{1/(\sigma-1)}}}_{n}(K),
$$
where the spaces $\mathcal{N}^{\sigma,\varepsilon}_{n}(K)$, $\varepsilon >0$,  are defined in Theorem \ref{main-GS}. Hence, Theorem \ref{main-GS}  follows from Theorem \ref{main-GSgen} with $M = M_\sigma$.
\end{example}
\begin{remark}
 The conditions \((M.2)\) and \((M.3)\) in   Theorem \ref{main-GS} are of a technical nature. We need to require them as our method  to show the implication \((3) \Rightarrow (1)\) relies on them; see Remark \ref{remark2} below for more details. While we believe that this implication should hold  without these assumptions, we are unable to prove it in that generality.
\end{remark}

The rest of this section is dedicated to the proof of Theorem \ref{main-GSgen}. We will show that  $(1) \Leftrightarrow (2)$, \((2) \Rightarrow (3)\), and  \((3) \Rightarrow (2)\) (under the additional assumptions $(M.2)$ and $(M.3)$ on $M$).   The equivalence $(1) \Leftrightarrow (2)$ in Theorem \ref{main-GSgen} will be a  consequence of Grothendieck's factorization theorem (see Theorem \ref{GFS} below). The implications   \((2) \Rightarrow (3)\) and  \((3) \Rightarrow (2)\) will be shown by applying Eidelheit's theorem in the same way as in the proof of Theorem \ref{main-pol}. 

 Let $M$ be a weight sequence and $h >0$. We endow  $\mathcal{S}^{M,h}(\R^d)$ with the locally convex topology induced by the system of norms $\{ \| \, \cdot \, \|^{M,h}_{n} \mid n \in \N \}$, where $\| \, \cdot \, \|^{M,h}_{n}$ is defined in \eqref{norm3}. Note that $\mathcal{S}^{M,h}(\R^d)$ is a Fr\'echet space. For $K \subseteq \R^d$ closed we write $\mathcal{S}^{M,h}_K = \{ f \in \mathcal{S}^{M,h}(\R^d) \mid  \operatorname{supp} f \subseteq K\}$ and endow this space with the subspace topology induced by $\mathcal{S}^{M,h}(\R^d)$. Since $\mathcal{S}^{M,h}_{K}$  is a  closed subspace of $\mathcal{S}^{M,h}(\R^d)$, it is also a Fr\'echet space. If $K$ is compact, we write $\mathcal{D}^{M,h}_{K} = \mathcal{S}^{M,h}_{K}$. Suppose $K$ has non-empty interior. Since $M$ is non-quasianalytic, the Denjoy-Carleman theorem  (see e.g. \cite[Theorem 4.2]{Komatsu})  implies that $\mathcal{D}^{M,h}_K$ is non-trivial.

 We now consider the analogues of  Propositions \ref{lemma1}-\ref{lemma3}.
\begin{proposition}\label{lemma1GS} Let $K \subseteq \R^d$ be a non-empty regular closed set, $M$ be  a weight sequence, $h>0$, and  $P \in \R[x_1, \ldots, x_d]$. The following statements are equivalent:
\begin{enumerate}[(1)]
\item $P \equiv 0$.
\item $P(x) = 0$ for all $x \in K$.
\item It holds that 
$$
\int_{\R^d} P(x) f(x) dx = 0, \qquad \forall f \in \mathcal{S}^{M,h}_{K}.
$$
\end{enumerate}
\end{proposition}
\begin{proof}
The implications $(1) \Rightarrow (2) \Rightarrow (3)$ are trivial. We now show $(3) \Rightarrow (1)$. Set
$$
\mathcal{D}^{(M)}_{\overline{B}(0,1)}= \bigcap_{h >0} \mathcal{D}^{M,h}_{\overline{B}(0,1)}.
$$
The Denjoy-Carleman theorem  implies that $\mathcal{D}^{(M)}_{\overline{B}(0,1)}$ is non-trivial. The result can now be shown in the same way as the implication  $(3) \Rightarrow (1)$ in Proposition \ref{lemma1}, but starting from an element $\varphi \in \mathcal{D}^{(M)}_{\overline{B}(0,1)}$ with $\int_{\R^d}\varphi(y)dy = 1$. 
\end{proof}


\begin{proposition}\label{lemma12}  Let $K \subseteq \R^d$ be a regular closed set, $M$ be a weight sequence, $h>0$, and $n \in \N$. For all $P \in \mathcal{N}^{M}_{h,n}(K)$ there is $C >0$ such that 
$$
\left| \int_{\R^d} P(x) f(x) dx \right| \leq C \| f \|^{M,h/d}_{n+d+1}, \qquad \forall f \in  \mathcal{S}^{M,h/d}_K.
$$
\end{proposition}
\begin{proof} Set $K_0 = \{ x \in \operatorname{int} K \, | \, d(x,\partial K) \leq 1\}$ and $m = n+d+1$. By reasoning as in the beginning of proof of Proposition \ref{lemma2},  it is clear that it suffices to show that there is $C_2 >0$ such that
$$
\sup_{x\in K_0} \frac{ |f(x)|(1 + |x|)^{m}}{\nu_M(hd(x,\partial K))}\leq C_2\| f \|^{M,h/d}_m,  \qquad \forall  f \in  \mathcal{S}^{M,h/d}_{K}.
$$
Let $ f \in  \mathcal{S}^{M,h/d}_{K}$ and  $x \in K_0$ be arbitrary. Pick $\xi \in \partial K$ such that $|x-\xi| = d(x,\partial K)$. Note that $f^{(\alpha)}(\xi) =0$ for all $\alpha \in \N^d$. Hence, Taylor's formula implies that for all $p \in \N$ there is $\gamma_p \in \{ tx + (1-t)\xi \mid t \in [0,1]\}$ such that
$$
f(x) = \sum_{|\alpha| = p} \frac{f^{(\alpha)}(\gamma_p)}{\alpha!} (x-\xi)^\alpha.
$$
As $x \in K_0$, it holds that $|x-\gamma_p| \leq |x-\xi| = d(x,\partial K) \leq 1$. Therefore,
$$
(1+|\gamma_p|)^{-m} \leq (1+|\gamma_p-x|)^m(1+|x|)^{-m} \leq  2^m(1+|x|)^{-m}.
$$ 
Since $\sum_{|\alpha| = p} \frac{1}{\alpha !} = \frac{d^p}{p!}$, we obtain that
$$
|f(x)| \leq   |x-\xi|^p \sum_{|\alpha| = p} \frac{|f^{(\alpha)}(\gamma_p)|}{\alpha!} \leq 2^m\| f \|^{M,h/d}_m\frac{1}{(1+|x|)^{m}} \frac{(hd(x,\partial K))^pM_p}{p!}.
$$
By taking the infimum over $p$, we find that
$$
|f(x)| \leq 2^m\| f \|^{M,h/d}_m\frac{1}{(1+|x|)^{m}} \inf_{p \in \N}  \frac{(hd(x,\partial K))^pM_p}{p!}=  2^m\| f \|^{M,h/d}_{m}\frac{\nu_M(hd(x,\partial K))}{(1+|x|)^{m}}.
$$
\end{proof}
\begin{proposition}\label{lemma3-ultra}  Let $K \subseteq \R^d$ be a regular closed set, $M$ be a weight sequence  satisfying $(M.2)$ and $(M.3)$, and $h >0$. There is $k >0$ such that for all $n \in \N$ the following property holds: If $P \in \R[x_1, \ldots, x_d]$ satisfies. for some $C >0$,
$$
\left| \int_{\R^d} P(x) f(x) dx \right| \leq C \| f \|^{M,h}_{n}, \qquad \forall f \in \mathcal{S}^{M,h}_K,
$$
then $P \in \mathcal{N}^{M,k}_{n}(K)$. 
\end{proposition}
The proof of Proposition \ref{lemma3-ultra} is based on the following variant of Lemma \ref{lemma131}. It will be more convenient to work with cubes instead of balls.  For $x \in \R^d$ and $r >0$ we write $Q(x,r) = x + [-r,r]^d$.

\begin{lemma}\label{lemma131-ultra} Let $M$ be  a weight sequence  satisfying $(M.2)$ and $(M.3)$, and $h >0$. There is $k > 0$ such that for all $m \in \N$ there is $C >0$ such that for all $P \in \R_m[x_1, \ldots, x_d]$, $x \in \R^d$, and  $0 < r \leq 1$
$$
|P(x)| \leq \frac{C}{\nu_M(kr)}\sup \left \{ \left | \int_{\R^d} P(y)\varphi(y) dy \right | \mid \varphi \in \mathcal{D}^{M,h}_{Q(x,r)}, \| \varphi\|^{M,h}_{0} \leq 1 \right \}.
$$
\end{lemma}
The simple scaling argument used in the proof of Lemma \ref{lemma131} does not apply in the ultradifferentiable setting. Instead, our proof of Lemma \ref{lemma131-ultra} relies on the existence of so-called optimal cutoff sequences first constructed Bruna \cite{Bruna}. We will use the version presented in  \cite{Rainer}.

\begin{lemma}\label{opt-cutoff}(cf.~\cite[Theorem 7.5]{Rainer}) Let $M$ be  a weight sequence  satisfying $(M.2)$ and $(M.3)$, and $h>0$. There are $C,k>0$ such that for every $r >0$ there exists $\theta \in \mathcal{D}^{M,h}_{Q(0,r/2)}$ satisfying 
\begin{itemize}
\item $0 \leq \theta \leq 1$.
\item $\theta= 1$ on $Q(0,r/4)$.
\item $\displaystyle \| \theta\|^{M,h}_0 \leq \frac{C}{\nu_M(kr)}$.
\end{itemize}
\end{lemma}
\begin{proof}
In \cite[Theorem 7.5]{Rainer}, this is shown for \(d = 1\) and for weight sequences \(M\) that satisfy \((M.2)\) and \((M.3)\), and are additionally strongly log-convex (meaning that the sequence \((M_p / p!)_{p \in \mathbb{N}}\) is log-convex). We now show that this result holds in arbitrary dimensions and without assuming that \(M\) is strongly log-convex; the latter is  indicated in \cite[Definition 7.3]{Rainer}.

By \cite[Lemma 4.7 and Corollary 4.8]{Rainer}, there exists a weight sequence $N$ such that $M \asymp N$, $N$ is strongly log-convex, and $N$  satisfies $(M.3)$. Moreover, it is clear that $N$ also satisfies $(M.2)$. As $N \subset M$, there are $C_0,h_0 >0$ such that   $\mathcal{D}^{N,h_0}_{[-r/2,r/2]} \subseteq  \mathcal{D}^{M,h}_{[-r/2,r/2]}$ and 
$$
 \| \varphi \|^{M,h}_0 \leq  C_0\| \varphi\|^{N,h_0}_0, \qquad \forall \varphi \in \mathcal{D}^{N,h_0}_{[-r/2,r/2]}.
$$
 As $M \subset N$, there are $C_1,h_1 >0$ such that
 $$
 \nu_M(t) \leq C_1 \nu_N(h_1t), \qquad \forall t \geq 0.
 $$
 By \cite[Theorem 7.5]{Rainer}, there is $k_0 >0$ such that for every $r >0$ there exists $\kappa \in \mathcal{D}^{N,h_0}_{[-r/2,r/2]}$ satisfying
$0 \leq \kappa \leq 1$,  $\kappa= 1$ on $[-r/4,r/4]$, and $\| \kappa\|^{N,h_0}_0 \leq 1/\nu_N(k_0r)$. Hence, $\kappa \in \mathcal{D}^{M,h}_{[-r/2,r/2]}$ and
$$
\| \kappa\|^{M,h}_0 \leq C_0\| \kappa \|^{N,h_0}_0 \leq \frac{C_0}{\nu_N(k_0r)} \leq \frac{C_0C_1}{\nu_M((k_0/h_1)r)}.
$$
Set 
$$
\theta(x_1,\ldots,x_d) = \kappa(x_1) \cdots \kappa(x_d).
$$
Then, $0 \leq \theta \leq 1$ and  $\theta= 1$ on $Q(0,r/4)$. Since $M$ is log-convex, it holds that
$$
M_{\alpha_1}\cdots M_{\alpha_d} \leq M_{\alpha_1+\cdots +\alpha_d} = M_\alpha, \qquad \forall \alpha= (\alpha_1,\ldots,\alpha_d) \in \N^d.
$$
Consequently, $\theta \in  \mathcal{D}^{M,h}_{Q(0,r/2)}$ and
$$
 \| \theta\|^{M,h}_0 \leq \left ( \| \kappa \|^{M,h}_0\right)^d \leq \frac{(C_0C_1)^d}{\nu_M((k_0/h_1)r)^d}.
$$
The result now follows from Lemma \ref{propnu}(1).
\end{proof}
We will need the following consequence of Lemma \ref{opt-cutoff}.
\begin{lemma}\label{opt-cutoff2} Let $M$ be  a weight sequence  satisfying $(M.2)$ and $(M.3)$, and $h>0$. There are $k,C >0$ such that for every $r >0$ there exists $\rho \in \mathcal{D}^{M,h}_{Q(0,r)}$ satisfying
\begin{itemize}
\item $\displaystyle \sum_{\lambda \in \Z^d} \rho(x- r\lambda) = 1$ for all $x \in \R^d$.
\item $\displaystyle \| \rho \|^{M,h}_0 \leq \frac{C}{\nu_M(kr)}$.
\end{itemize}
\end{lemma}
\begin{proof}
Let $C,k>0$ be as in Lemma \ref{opt-cutoff}. Let $r >0$ be arbitrary and choose $\theta \in \mathcal{D}^{M,h}_{Q(0,r/2)}$ such that  $0 \leq \theta \leq 1$,  $\theta= 1$ on $Q(0,r/4)$, and  $\displaystyle \| \theta\|^{M,h}_0 \leq C/\nu_M(kr)$. Set $C_0 =  \int_{\R^d}\theta(y)dy >0$ and define
$$
\rho(x) = \frac{1}{C_0} \int_{Q(0,r/2)} \theta(x+y) dy, \qquad x \in \R^d.
$$
It holds that  $\rho \in \mathcal{D}^{M,h}_{Q(0,r)}$ and 
$$
\sum_{\lambda \in \Z^d} \rho(x- r\lambda) = 1, \qquad \forall x \in \R^d.
$$
As $\theta \geq 0$ and  $\theta= 1$ on $Q(0,r/4)$, we find that $C_0 \geq (r/4)^d$. Hence,
$$
 \| \rho \|^{M,h}_0 \leq \frac{1}{C_0} (r/2)^d  \| \theta \|^{M,h}_0 \leq \frac{2^dC}{\nu_M(kr)}.
$$
\end{proof}

\begin{proof}[Proof of Lemma \ref{lemma131-ultra}]
By a simple translation argument,  it suffices to consider the case $x = 0$.  Since $M$ satisfies $(M.2)$, there is $C_0 >0$ such that
\begin{equation}\label{M22}
M_{p+q} \leq C_0^{p+q}M_pM_q, \qquad \forall p,q \in \N.
\end{equation}
Set $h_0 = h/(2C_0)$. We write
$$
\|P\|_\infty = \sup_{y \in Q(0,1)} |P(y)|, \qquad P \in \R_m[x_1, \ldots, x_d],
$$
and
$$
\|P\| = \sup \left \{ \left | \int_{\R^d} P(y) \psi(y) dy \right | \mid \psi \in \mathcal{D}^{M,h_0}_{Q(0,1)}, \| \psi \|^{M,h_0}_{0} \leq 1 \right \}, \qquad P \in \R_m[x_1, \ldots, x_d].
$$
By Proposition \ref{lemma1GS} (with $K = Q(0,1)$), $\|\, \cdot \,\|_\infty$ and $\|\, \cdot \,\|$ are both norms on  $\R_m[x_1, \ldots, x_d]$. Since the space $\R_m[x_1, \ldots, x_d]$ is finite-dimensional, $\|\, \cdot \, \|_\infty$ and $\|\, \cdot \, \|$ are equivalent. Hence, there is $C >0$ such that 
$$
\|P\|_\infty \leq C \|P\|, \qquad \forall P \in \R_m[x_1, \ldots, x_d].
$$
 By Lemma \ref{opt-cutoff2}, there are  $C_1,k_0 >0$ such that for every $r >0$ there exists $\rho \in \mathcal{D}^{M,h_0}_{Q(0,r)}$ such that $\sum_{\lambda \in \Z^d} \rho(x- r\lambda) = 1$ for all $x \in \R^d$ and  $\| \rho \|^{M,h_0}_0 \leq C_1/\nu_M(k_0r)$.
Let $0 < r \leq 1$ and $P \in \R_m[x_1,\ldots,x_d]$ be arbitrary.  Choose $\rho$ as above.  Note that 
$$
|r\lambda| > (1+r)\sqrt{d} \, \, \Longrightarrow \,\,   Q(r\lambda,r) \cap Q(0,1) = \emptyset, \qquad \forall \lambda \in \Z^d.
$$
We find that, for all $\psi \in \mathcal{D}^{M,h_0}_{Q(0,1)}$,
\begin{align*}
\left | \int_{\R^d} P(y) \psi(y) dy \right | &\leq  \sum_{\lambda \in \Z^d} \left | \int_{\R^d} P(y) \rho(y-r\lambda) \psi(y) dy \right |  \\
&= \sum_{\substack{ \lambda \in \mathbb{Z}^d, \\ |r\lambda| \leq (1+r)\sqrt{d} }}\left | \int_{\R^d} P(y+r\lambda) \rho(y) \psi(y+r\lambda) dy \right |  \\
&\leq  \sum_{\substack{ \lambda \in \mathbb{Z}^d, \\ |r\lambda| \leq 2\sqrt{d} }} \sum_{|\alpha| \leq m} \frac{(r\lambda)^\alpha}{\alpha!}\left | \int_{\R^d} P^{(\alpha)}(y) \rho(y) \psi(y+r\lambda) dy \right | \\
&\leq   \left(\frac{4e^2\sqrt{d}}{r}\right)^d \sup_{\lambda \in \Z^d} \max_{|\alpha| \leq m} \left | \int_{\R^d} P(y) (\rho(y) \psi(y+r\lambda))^{(\alpha)}dy \right |.  
\end{align*}
Let  $\lambda \in \Z^d$ and $\alpha \in \N^d$ with $|\alpha| \leq m$ be arbitrary. Note that $\operatorname{supp} (\rho(\, \cdot \,) \psi(\, \cdot \, + r\lambda)) \subseteq Q(0,r)$. As  $M$ is log-convex, it holds that
$$
M_pM_q \leq M_{p+q}, \qquad \forall p,q \in \N.
$$
By using this inequality and  \eqref{M22}, we obtain that
\begin{align*}
 \| (\rho(\, \cdot \,) \psi(\, \cdot \, + r\lambda))^{(\alpha)}\|^{M,h}_0 &= \sup_{\beta \in \N^d} \sup_{x \in \R^d} \frac{|(\rho(y)\psi(y+r\lambda))^{(\alpha+\beta)}|}{h^{|\beta|}M_\beta}\\
 &\leq \sup_{\beta \in \N^d} \sup_{x \in \R^d}  \frac{1}{h^{|\beta|}M_\beta}  \sum_{\gamma \leq \alpha + \beta} \binom{\alpha+\beta}{\gamma} |\rho^{(\gamma)}(y)| |\psi^{(\alpha + \beta - \gamma)}(y+r\lambda)| \\
 &\leq  \| \rho \|^{M,h_0}_{0} \| \psi \|^{M,h_0}_{0} \sup_{\beta \in \N^d} \frac{h_0^{|\alpha|+|\beta|}}{h^{|\beta|}M_\beta} \sum_{\gamma \leq \alpha + \beta} \binom{\alpha+\beta}{\gamma} M_\gamma M_{\alpha+\beta-\gamma}\\
 &\leq  \frac{C_1\| \psi \|^{M,h_0}_{0}}{{\nu_M(k_0r)}}  \sup_{\beta \in \N^d} \frac{(2h_0)^{|\alpha|+|\beta|}M_{\alpha+\beta}}{h^{|\beta|}M_\beta} \\
 &\leq  \frac{C_1\| \psi \|^{M,h_0}_{0}}{{\nu_M(k_0r)}} \sup_{\beta \in \N^d} \frac{(2h_0C_0)^{|\alpha|+|\beta|}M_\alpha}{h^{|\beta|}} \\
& \leq  C_1\max\{1,h^m\}M_m \frac{\| \psi \|^{M,h_0}_{0}}{{\nu_M(k_0r)}}.
\end{align*}
Set $C_2 = (4e^2\sqrt{d})^dC_1\max\{1,h^m\}M_m$. We find that 
\begin{align*}
&|P(0)| \leq \|P\|_\infty \leq C \|P\| \\
 &= C\sup \left \{ \left | \int_{\R^d} P(y) \psi(y) dy \right | \mid \psi \in \mathcal{D}^{M,h_0}_{Q(0,1)}, \| \psi \|^{M,h_0}_{0} \leq 1 \right \} \\
 &\leq \left(\frac{4e^2\sqrt{d}}{r}\right)^dC \sup \left \{ \left | \int_{\R^d} P(y) (\rho(y) \psi(y+r\nu))^{(\alpha)}dy \right | \mid \lambda \in \Z^d, |\alpha| \leq m, \psi \in \mathcal{D}^{M,h_0}_{Q(0,1)}, \| \psi \|^{M,h_0}_{0} \leq 1 \right \} \\
  &\leq \frac{CC_2}{r^d\nu_M(k_0r)} \left  \{ \left | \int_{\R^d} P(y) \varphi(y)dy \right | \mid  \varphi \in \mathcal{D}^{M,h}_{Q(0,r)}, \| \varphi \|^{M,h}_{0} \leq 1 \right \}.
\end{align*}
The result now follows from Lemma \ref{propnu}(2).
\end{proof}

\begin{proof}[Proof of Proposition \ref{lemma3-ultra}]
The result can be shown in the same way as Proposition \ref{lemma3}, but now using  Lemma \ref{lemma131-ultra} instead of Lemma \ref{lemma131}.
\end{proof}

Finally, we present Grothendieck's factorization theorem.
\begin{theorem}[Grothendieck's factorization theorem] \cite[Theorem 24.33]{MV}\label{GFS} Let $E_n$, $n \in \N$, be Fr\'echet spaces, $F$ be a Hausdorff locally convex space, and $T_n: E_n \to F$, $n \in \N$, be continuous linear mappings. If 
$$
F= \bigcup_{n \in \N} T_n(E_n),
$$
there exists $n_0 \in \N$ such that $T_{n_0}: E_{n_0} \to F$ is surjective.
\end{theorem}

\begin{proof}[Proof of Theorem \ref{main-GSgen}]
$(1) \Rightarrow (2)$: For $n \in \N$ we define the continuous linear mappings
$$
T_n: \mathcal{S}^{M,n}_K \to \C^{\N^d}, \, f \to \left ( \int_{\R^d} x^\alpha f(x) dx \right)_{\alpha \in \N^d}.
$$
Since the unrestricted $K$-moment problem is solvable in $\mathcal{S}^{\{M\}}(\R^d)$, it holds that
$$
\C^{\N^d} = \bigcup_{n \in \N} T_n(\mathcal{S}^{M,n}_K).
$$
As $\mathcal{S}^{M,n}_K$, $n \in \N$, are Fr\'echet spaces, Grothendieck's factorization theorem implies that there exists $n_0 \in \N$ such that $T_{n_0}: \mathcal{S}^{M,n_0}_K \to \C^{\N^d}$ is surjective. This means that the unrestricted $K$-moment problem is solvable in $\mathcal{S}^{M,n_0}(\R^d)$. \\
$(1) \Rightarrow (2)$: Trivial. \\
$(2) \Rightarrow (3)$ and $(3) \Rightarrow (2)$ (under the additional assumptions $(M.2)$ and $(M.3)$ on $M$):
For $P \in \R[x_1,\ldots, x_d]$ and $h >0$ we define $ e'_P \in (\mathcal{S}^{M,h}_{K})'$ via
$$
\langle e'_P, f \rangle = \int_{\R^d} P(x) f(x)dx, \qquad f \in \mathcal{S}^{M,h}_{K}.
$$
Clearly, the unrestricted $K$-moment problem is solvable in $\mathcal{S}^{M,h}(\R^d)$ if and only if for every  $(c_\alpha)_{\alpha \in \N^d} \in \R^{\N^d}$ there exists $f \in \mathcal{S}^{M,h}_{K}$  such that 
$$
\langle e'_{x^\alpha}, f \rangle = c_\alpha, \qquad \forall \alpha \in \N^d.
$$
Proposition \ref{lemma1GS} implies that the family $(e'_{x^\alpha})_{\alpha \in \N^d} \subseteq  (\mathcal{S}^{M,h}_{K})'$ is linearly independent.
 Consider the following two statements:
\begin{enumerate}
\item[(a)] There is $h >0$ such that $(\operatorname{span} \{e'_{x^\alpha} \mid \alpha \in \N^d\})_{\| \, \cdot \, \|^{M,h}_{n}} \subseteq  (\mathcal{S}^{M,h}_{K})'$ is finite-dimensional for all $n \in \N$.
\item[(b)] There is $h >0$ such that $\mathcal{N}^{M,h}_{n}(K)$ is finite-dimensional for all $n \in \N$.
\end{enumerate}
As $\mathcal{S}^{M,h}_K$, $h >0$, are Fr\'echet spaces,  Eidelheit's theorem yields that it suffices to show that  (a) $\Rightarrow$ (b) and (b) $\Rightarrow$ (a)  (under the additional assumptions $(M.2)$ and $(M.3)$ on $M$). This can be done in the same way as in the proof of Theorem \ref{main-pol}, but now using Propositions  \ref{lemma12} and \ref{lemma3-ultra} instead Propositions  \ref{lemma2} and \ref{lemma3}.
\end{proof}

\begin{remark}\label{remark2}
Let \(M\) be a strongly log-convex weight sequence satisfying \((M.2)\). Condition \((M.3)\) is equivalent to the existence of optimal cutoff sequences as in Lemma \ref{lemma131-ultra}; see \cite[Theorem 7.5]{Rainer} for a precise statement of this fact. Hence, the current proof method cannot establish the implication \((3) \Rightarrow (2)\) in Theorem \ref{main-GSgen} without assuming \((M.3)\).
\end{remark}

\section{Examples}\label{sect-examples}
In this final section, we give some necessary and sufficient geometric conditions on regular closed sets $K \subseteq \R^d$ such that the  unrestricted $K$-moment problem is solvable in $\mathcal{S}(\R^d)$ and $\mathcal{S}^{\{M\}}(\R^d)$. We also discuss several examples. We start with a necessary condition.
\begin{proposition} \label{nec}
Let $K \subseteq \R^d$ be a regular closed set.
\begin{enumerate}
\item If the unrestricted $K$-moment problem is solvable in $\mathcal{S}(\R^d)$, then for all 
$k \in \N$ there is $l >0$ such that for each
$i =1, \ldots,d$
$$
\sup_{x = (x_1, \ldots,x_d) \in K} |x_i|^l d_K(x)^k = \infty.
$$
\item Let $M$ be a weight sequence. Suppose that the unrestricted $K$-moment problem is solvable in $\mathcal{S}^{\{M\}}(\R^d)$. There are $h,l >0$  such that for each $i = 1, \ldots,d$
\begin{equation}
\label{reduction}
\sup_{x = (x_1, \ldots,x_d)  \in K} |x_i|^l  \nu_M(hd_K(x)) = \infty.
\end{equation}
Moreover, if $M$ additionally satisfies $(M.3)$, then for all $h >0$  there is $l >0$ such that \eqref{reduction} holds for all $i = 1, \ldots,d$.
\end{enumerate}
\end{proposition}
\begin{proof}
We only show (2), as the argument for (1) is similar. Since the unrestricted $K$-moment problem is solvable in $\mathcal{S}^{\{M\}}(\R^d)$, Theorem \ref{main-GSgen} implies that there is $h >0$ such that  $\mathcal{N}^{M,h}_0(K)$ is finite-dimensional. Hence, there is $m \in \N$ such that $\mathcal{N}^{M,h}_0(K) \subseteq \R_m[x_1,\ldots,x_d]$.  In particular, for each $i = 1, \ldots,d$, the polynomial $P(x) = x_i^{m+1}$ does not belong to $\mathcal{N}^{M,h}_0(K)$. This means that
$$
\sup_{x  \in K} |x_i|^{m+1} \nu_M(hd_K(x)) = \sup_{x  \in K} |P(x)| \nu_M(hd_K(x)) = \infty,
$$ 
which is \eqref{reduction} with $l = m+1$. 

Next, assume that $M$ satisfies $(M.3)$ and  let $h,l >0$ be such that \eqref{reduction} holds for all $i = 1, \ldots, d$. Let $k >0$ be arbitrary. By Lemma \ref{propnu}(3), there are $C_0,C_1 >0$ such that  
$$
\nu_M(ht)^{C_0} \leq C_1\nu_M(kt), \qquad \forall t \geq 0. 
$$
Hence, for all  $i=1, \ldots, d$,
$$
\sup_{x \in K}|x_i|^{C_0l} \nu_M(kd_K(x)) \geq \frac{1}{C_1} \sup_{x \in K} \left (|x_i|^{l} \nu_M(hd_K(x))\right)^{C_0} = \infty.
$$
  \end{proof}
We now consider the one-dimensional case. Let $K \subseteq \R$ be closed. Theorem \ref{main-schmudgen} implies that the unrestricted $K$-moment problem is solvable in $\mathcal{M}_{\operatorname{pol}}(\R)$ if and only if $K$ is unbounded (see also \cite[Corollary 6]{Schmudgen}), which was first shown by Polya \cite{Polya}. In our setting, we obtain the following  characterization:
\begin{proposition} \label{dim1}
Let $K \subseteq \R$ be a regular closed set.
\begin{enumerate}
\item The following statements are equivalent:
\begin{itemize}
\item[(a)] The unrestricted $K$-moment problem is solvable in $\mathcal{S}(\R)$. 
\item[(b)] There is  $l >0$ such that 
$$
\sup_{x \in K} |x|^l d_K(x) = \infty.
$$
\end{itemize}
\item Let $M$ be a weight sequence satisfying $(M.2)$ and $(M.3)$. The following statements are equivalent:
\begin{itemize}
\item[(a)] The unrestricted $K$-moment problem is solvable in $\mathcal{S}^{\{M\}}(\R)$.
\item[(b)] There is $l >0$   such that
\begin{equation}
\label{noodz??}
\sup_{x \in K} |x|^l \nu_M(d_K(x)) = \infty.
\end{equation}
\end{itemize}
\end{enumerate}
\end{proposition}
\begin{proof}
We only show (2), as the argument for (1) is similar once we observe that condition (b) in (1) implies that, for all $k \in \N$,
$$
\sup_{x \in K} |x|^{kl} d_K(x)^k = \infty.
$$
The implication (a) $\Rightarrow$ (b) has been shown in Proposition \ref{nec}. We now prove (b) $\Rightarrow$ (a). By Theorem \ref{main-GSgen}, it suffices to show that $\mathcal{N}^{M,1}_n(K) \subseteq \R_{\lfloor l +n \rfloor}[x]$ for all $n \in \N$. Let $n \in \N$ and $P \in  \mathcal{N}^{M,1}_n(K)$ be arbitrary. Set $m = \deg P$.  There are $C_0 >0$ and $R \geq 1$ such that
 $$
 |P(x)| \geq C_0|x|^{m}, \qquad \forall |x| \geq R.
 $$
 Since $P \in \mathcal{N}^{M,1}_n(K)$, there is $C_1 >0$ such that
 $$
 |P(x)| \leq C_1\frac{(1+|x|)^n}{\nu_M(d_K(x))} \leq 2^nC_1\frac{|x|^n}{\nu_M(d_K(x))}, \qquad \forall x \in K, |x| \geq 1.
 $$
By combining these two inequalities, we find that
$$
|x|^{m-n}\nu_M(d_K(x)) \leq \frac{2^nC_1}{C_0}, \qquad \forall x \in K, |x| \geq R.
$$
Condition \eqref{noodz??} implies that $m-n < l$ and thus $m \leq \lfloor l +n \rfloor$.
\end{proof}

Next, we present a necessary condition in the case $d \geq 2$. It is much inspired by \cite[Corollary 7]{Schmudgen}.

\begin{proposition} \label{suff}
Let $K \subseteq \R^d$ be a regular closed set.
\begin{enumerate}
\item Suppose that, for each $i =1, \ldots, d$, there exists a Zariski dense subset $K_i$ of $\R^{d-1}$ and $l_i >0$ such that for all $y = (y_1, \ldots, y_{i-1},y_{i+1}, \ldots y_d) \in K_i$ it holds that
$$
\sup_{x \in K \cap (\{y\} \times_i \R)} |x|^{l_i} d_K(x) = \infty,
$$
where $\{y\} \times_i \R = \{ (y_1, \ldots, y_{i-1},t,y_{i+1}, \ldots y_d) \mid t \in \R \}$. Then, the unrestricted $K$-moment problem is solvable in $\mathcal{S}(\R^d)$.

\item Let $M$ be a weight sequence satisfying $(M.2)$ and $(M.3)$. Suppose that, for each $i =1, \ldots, d$, there exists a Zariski dense subset $K_i$ of $\R^{d-1}$ and $l_i >0$ such that for each $y = (y_1, \ldots, y_{i-1},y_{i+1}, \ldots y_d) \in K_i$ it holds that
\begin{equation}
\label{cond2sch}
\sup_{x \in K \cap (\{y\} \times_i \R)} |x|^{l_i} \nu_M(d_K(x)) = \infty.
\end{equation}
Then, the unrestricted $K$-moment problem is solvable in $\mathcal{S}^{\{M\}}(\R^d)$.
\end{enumerate}
\end{proposition}
\begin{proof}  We only show (2), as the argument for (1) is similar once we observe that, for all $k \in \N$,
$$
\sup_{x \in K \cap (\{y\} \times_i \R)} |x|^{kl_i} d_K(x)^k = \infty.
$$
We claim that for all \( P \in \mathcal{N}^{M,1}_n(K) \), \( n \in \mathbb{N} \), the degree of \( P \) with respect to \( x_i \), \( i = 1, \ldots, d \), is at most \( \lfloor l_i +n \rfloor\). This implies that \( P \in \mathbb{R}_{dn +\lfloor l_1\rfloor + \cdots + \lfloor l_d\rfloor}[x_1, \ldots, x_d] \) and the result follows from Theorem \ref{main-GSgen}. We now prove the claim. Let $n \in \N$, $P \in \mathcal{N}^{M,1}_n(K)$, and $i \in \{1, \ldots,d\}$ be arbitrary. Set $m_i  =\deg_{x_i} P$. There are $P_j \in \R[x_1, \ldots,x_{i-1},x_{i+1},\ldots x_d]$, $j=0,\ldots,m_i$, such that $P_{m_i} \not \equiv 0$ and 
$$
P(x) = \sum_{j = 0}^{m_i} P_j(x_1, \ldots,x_{i-1},x_{i+1},\ldots, x_d)x_i^j.
$$
Since $K_i$ is Zariski dense in $\R^{d-1}$ and $P_{m_i} \not \equiv 0$, there is $y = (y_1, \ldots, y_{i-1},y_{i+1}, \ldots y_d) \in K_i$ such that $P_{m_i}(y) \neq 0$. For $t \in \R$ we write $y_t =  (y_1, \ldots, y_{i-1},t, y_{i+1}, \ldots y_d) \in \R^d$. Note that
$$
P(y_t) = \sum_{j = 0}^{m_i} P_j(y)t^j, \qquad t \in \R.
$$
 As $P_{m_i}(y) \neq 0$, there is $C_0 >0$ and $R \geq 1$ such that
 $$
 |P(y_t)| \geq C_0|t|^{m_i}, \qquad \forall |t| \geq R.
 $$
 Hence,
\begin{equation}
\label{cond11}
  |P(y_t)| \geq \frac{C_0}{2^{m_i/2}}|y_t|^{m_i}, \qquad \forall t \geq \max\{ |y|, R\}.
\end{equation}
 Since $P \in \mathcal{N}^{M,1}_n(K)$, there is $C_1 >0$ such that
\begin{equation}
\label{cond22}
 |P(x)| \leq C_1\frac{(1+|x|)^n}{\nu_M(d_K(x))} \leq  2^nC_1\frac{|x|^n}{\nu_M(d_K(x))}, \qquad \forall x \in K, |x| \geq 1.
\end{equation}
 By combining \eqref{cond11} and \eqref{cond22} with $x =y_t$, we find that
$$
|y_t|^{m_i-n}\nu_M(d_K(y_t)) \leq \frac{2^{n+m_i/2}C_1}{C_0}, \qquad \forall y_t \in K, t \geq \max\{ |y|, R\}.
$$
Condition \eqref{cond2sch} implies that $m_i-n < l_i$ and thus $m_i \leq \lfloor l_i +n \rfloor$.
 \end{proof}
 
 \begin{example}\label{ex-1}
Let $M$ be a weight sequence satisfying $(M.2)$ and $(M.3)$. Proposition \ref{suff}(2) implies that the unrestricted $K$-moment problem is solvable in $\mathcal{S}^{\{M\}}(\R^d)$ for $K = [0,\infty)^d$. 
 \end{example}
 We make the following simple  but useful observation:
 \begin{lemma} \label{transf}
Let $K \subseteq \R^d$ be a regular closed set and $A: \R^d \to \R^d$ be a bijective linear transformation.
\begin{enumerate}
\item The unrestricted $K$-moment problem is solvable in $\mathcal{S}(\R^d)$ if and only if  the unrestricted $A(K)$-moment problem is so.
\item Let $M$ be a weight sequence. The unrestricted $K$-moment problem is solvable in $\mathcal{S}^{\{M\}}(\R^d)$ if and only if  the unrestricted $A(K)$-moment problem is so.
 \end{enumerate}
 \end{lemma}
 \begin{proof}
 This can be shown by using the same argument as in the proof of \cite[Theorem 4.2]{Estrada}\footnote{This also easily follows from Propositions \ref{main-pol} and \ref{main-GSgen}; however, in the ultradifferentiable case, one would unnecessarily need to impose  $(M.2)$ and $(M.3)$ on $M$.}.
 \end{proof}
  \begin{example}
  Recall that a set $\Gamma \subseteq \R^d$ is called  a \emph{cone} (with vertex at $0$) if $\lambda x \in \Gamma$ for all $x \in \Gamma$ and $\lambda >0$. Let $\Gamma$ be a non-empty regular closed cone. As $\Gamma$ has non-empty interior, there is a 
 bijective linear transformation $A: \R^d \to \R^d$ such that $[0,\infty)^d \subseteq A(\Gamma)$. Hence, Example \ref{ex-1} and Lemma \ref{transf}(2) imply that  the unrestricted $\Gamma$-moment problem is solvable in $\mathcal{S}^{\{M\}}(\R^d)$ for any weight sequence $M$ satisfying $(M.2)$ and $(M.3)$.
   \end{example}
Let $a = (a_j)_{j \in \N}$ and $b = (b_j)_{j \in \N}$ be two real sequences satisfying
\begin{equation}
\label{seqcond}
0 \leq  a_1 < b_1 < a_2 < b_2 < \cdots.
\end{equation}
We define, for $d =1$,
$$
K_{a,b} = \bigcup_{j \in \N} [a_j,b_j]
$$
and, for $d >1$, 
$$
K_{a,b} = \bigcup_{j \in \N} [a_j,b_j] \times \R^{d-1}.
$$
We now characterize when the $K_{a,b}$-unrestricted moment problem is solvable in $\mathcal{S}(\R^d)$ and  $\mathcal{S}^{\{M\}}(\R^d)$ in terms of the sequences $a$ and $b$. 


\begin{proposition}\label{KabSGS}
Let $a = (a_j)_{j \in \N}$ and $b = (b_j)_{j \in \N}$ be two real sequences satisfying \eqref{seqcond}.  
\begin{enumerate}
\item The following statements are equivalent:
\begin{itemize}
\item[(a)]  The unrestricted $K_{a,b}$-moment problem is solvable in $\mathcal{S}(\R^d)$.
\item[(b)] There is  $l >0$ such that 
$$
\sup_{j \in \N} a_j^l (b_j-a_j) = \infty.
$$
\end{itemize}
\item Let $M$ be a weight sequence satisfying $(M.2)$ and $(M.3)$.  The following statements are equivalent:
\begin{itemize}
\item[(a)]  The unrestricted $K_{a,b}$-moment problem is solvable in $\mathcal{S}^{\{M\}}(\R^d)$. 
\item[(b)]  There is  $l >0$ such that  
\begin{equation}
\label{yetagain}
\sup_{j \in \N} a_j^l \nu_M(b_j-a_j) = \infty.
\end{equation}
\end{itemize}
\end{enumerate}
\end{proposition}
\begin{proof}
We only prove (2), as the argument for (1) is similar. Furthermore, we suppose that \( d > 1 \); the case \( d = 1 \) is analogous, but one needs to use Proposition \ref{dim1} instead of Proposition  \ref{suff}. We write points in \( \mathbb{R}^d = \mathbb{R} \times \mathbb{R}^{d-1} \) as \( (t, x) \). \\
(a) $\Rightarrow$ (b): Suppose that there exists a subsequence $(a_{j_k})_{k \in \N}$ such that $(b_{j_k} - a_{j_k})/2 \geq 1$ for all $k \in \N$. Then, \eqref{yetagain} holds with $l =1$.
Therefore, we may assume that there is $j_0 \in \N$ such that 
\begin{equation}
\label{small}
\frac{b_j-a_j}{2} \leq 1, \qquad \forall j \geq j_0. 
\end{equation}
Proposition \ref{nec}(2) implies that there is $l >0$ such that
\begin{align}
\label{toomuchlabels}
\infty &= \sup_{(t,x) \in K_{a,b}} |t|^l \nu_M(d_{K_{a,b}}(t,x)) \\
&= \sup_{j \in \mathbb{N}} \sup_{t \in [a_j, b_j]} |t|^l \nu_M(d_{[a_j, b_j]}(t)) \notag \\
&\leq  \max_{j <j_0} \sup_{t \in [a_j, b_j]} |t|^l \nu_M(d_{[a_j, b_j]}(t)) + \sup_{j \geq j_0} b_j^l \nu_M(b_j-a_j). \notag
\end{align}
In particular, $b_{j} \to \infty$ and thus also $a_{j} \to \infty$. By \eqref{small}, there is $j_1 \geq j_0$ such that $b_j \leq 2 a_j$ for all $j \geq j_1$. Hence, the result follows from \eqref{toomuchlabels}. \\
(b) $\Rightarrow$ (a): Set $\delta_j = \min \{1, (b_j-a_j)/2\}$ for $j \in \N$. Condition \eqref{yetagain} and Lemma \ref{propnu}(3) imply that there is $k >0$ such that
\begin{equation}
\label{yetagain2}
\sup_{j \in \N} a_j^k \nu_M(\delta_j) = \infty.
\end{equation}
We wish to apply Proposition \ref{suff}. Condition  \eqref{yetagain2} implies that $a_j \to \infty$. Fix $c_j \in (a_j, b_j)$ for all $j \in \N$ and note that $c_j \to \infty$. Set  $M_1 = \R^{d-1}$ and $M_i =  \bigcup_{j \in \N} \{c_j\} \times \R^{d-2}$ for $i =2, \ldots, d$. The sets $M_i $, $i =1, \ldots, 2$, are Zariski dense in $\R^{d-1}$ (for $i \geq 2$ this follows from $c_j \to \infty$). Condition \eqref{yetagain2} implies that, for all  $y \in M_1 = \R^{d-1}$,
\begin{align*}
\sup_{x \in K_{a,b} \cap (\{y\} \times_1 \R)} |x|^{k} \nu_M(d_{K_{a,b}}(x)) &= \sup_{(t,y) \in K_{a,b}} |(t,y)|^{k} \nu_M(d_{K_{a,b}}((t,y)))  \\
&\geq  \sup_{j \in \N}\sup_{t \in [a_j,b_j]} |t|^k \nu_M(d_{[a_j,b_j]}(t))  \\
&\geq \sup_{j \in \N} a_j^k \nu_M(\delta_j) = \infty. 
\end{align*}
Let $i \in \{2, \ldots ,d\}$ and $y \in M_i$ be arbitrary. There is $j \in \N$ such that $y \in \{c_j\} \times \R^{d-2}$. Then,
$$
\sup_{x \in K_{a,b} \cap (\{y\} \times_i \R)} |x| \nu_M(d_{K_{a,b}}(x))  \geq \sup_{t \in \R} |t| \nu_M(d_{[a_j,b_j]}(c_j)) = \infty. 
$$
The result now follows from Proposition \ref{suff}.
\end{proof}

\begin{examples}  \label{example-last}

\begin{enumerate}[(1)]
\item  Let $a = (\log(1+j)^s)_{j \in \N}$, $s>0$, and $b = (b_j)_{j \in \N}$ be any real sequence such that  \eqref{seqcond} is satisfied. Then, the unrestricted $K_{a,b}$-moment problem is not solvable in $\mathcal{S}(\R^d)$. This follows from Proposition \ref{KabSGS}(1) and the fact that
$$
b_j-a_j \leq a_{j+1} -a_j = O \left ( \frac{\log(1+j)^{s-1}}{j} \right).
$$
\item Let $a = (j^s)_{j \in \N}$, $s>0$, and $b = (b_j)_{j \in \N}$ be a real sequence such that  \eqref{seqcond} is satisfied. By Proposition \ref{KabSGS}(1), the unrestricted $K_{a,b}$-moment problem is solvable in $\mathcal{S}(\R^d)$ if and only if there is $l >0$ such that
$$
\sup_{j \in \N} j^l(b_j-j^s) = \infty. 
$$
\item Let $\sigma >1$. Let $a = (j^s)_{j \in \N}$, $s>0$, and $b = (b_j)_{j \in \N}$ be a real sequence such that  \eqref{seqcond} is satisfied. By Proposition \ref{KabSGS}(2) and Example \ref{weightG}, the unrestricted $K_{a,b}$-moment problem is solvable in $\mathcal{S}^{\sigma}(\R^d)$ if and only if there is $l >0$ such that
$$
\sup_{j \in \N} j^le^{-\left(\frac{1}{(b_j-j^s)}\right)^{\frac{1}{\sigma-1}}}= \infty. 
$$
For $b_r = \left(j + \left(\frac{1}{\log(e+j^s)}\right)^{r-1}\right)_{j \in \N}$, $r >1$, this condition is satisfied if and only if $r \leq \sigma$.  
Hence, the unrestricted  $K_{a,b_r}$-moment problem is solvable in $\mathcal{S}^{\sigma}(\R^d)$ if and only if $r \leq \sigma$. Note that  this implies Theorem \ref{KGSGS}.
\end{enumerate}
\end{examples}
%
Finally, we give a version of Theorem \ref{KGSGS} for general weight sequences:
\begin{theorem}\label{KGSGSgen}
Let $M$ and $N$ be two weight sequences satisfying $(M.2)$ and $(M.3)$. Suppose that $N \prec M$. There are sequences $a = (a_j)_{j \in \N}$ and $b = (b_j)_{j \in \N}$ satisfying \eqref{seqcond} such that the unrestricted $K_{a,b}$-moment problem is solvable in $\mathcal{S}^{\{M\}}(\R^d)$ but not in $\mathcal{S}^{\{N\}} (\R^d)$. 
\end{theorem}
\begin{proof}
We use the same idea as in Example  \ref{example-last}(3). Choose $j_0 \in \N$ such that $1/j_0 < \nu_M(1)$. Choose $\varepsilon_j >0$ such that $\nu_M(\varepsilon_j) = 1/j$ for $j \geq j_0$ and set $\varepsilon_j = 1/2$ for $j < j_0$. Then, $\varepsilon_j <1$ for all $j \in \N$. Define $a=(j)_{j \in \N}$ and $b = (j+\varepsilon_j)_{j \in \N}$. Note that $a$ and $b$ satisfy \eqref{seqcond}. We claim that the $K_{a,b}$-moment problem is solvable in $\mathcal{S}^{\{M\}}(\R^d)$ but not in $\mathcal{S}^{\{N\}} (\R^d)$.  By Proposition \ref{KabSGS}(2), it suffices to show that
$$
\sup_{j \in \N} j^2 \nu_M(\varepsilon_j) = \infty
$$
and that, for all $l >0$,
$$
\sup_{j \in \N} j^l \nu_N(\varepsilon_j) <\infty.
$$
The former property follows directly from the fact that  $\nu_M(\varepsilon_j) = 1/j$ for $j \geq j_0$. We now show the latter. Let $l >0$ be arbitrary. Lemma \ref{propnu}(1) implies that there is $C_0 >0$ such that
$$
\nu_M(t) \leq \nu_M(C_0t)^l, \qquad t \geq 0.
$$
Since $N \prec M$, there is $C_1 >0$ such that
$$
\nu_N(t) \leq C_1 \nu_M(t/C_0), \qquad t \geq 0.
$$
 Since $\nu_M(\varepsilon_j) = 1/j$ for $j \geq j_0$, we obtain that
$$ j^l \nu_N(\varepsilon_j) = \frac{\nu_N(\varepsilon_j)}{\nu_M(\varepsilon_j)^l} \leq\frac{\nu_N(\varepsilon_j)}{\nu_M(\varepsilon_j/C_0)} \leq C_1, \qquad \forall j \geq j_0.
$$
\end{proof}


\end{document}